\documentclass[12]{amsart}
\usepackage{amsfonts}
\usepackage{mathrsfs}
\usepackage{amsmath}
\usepackage{amssymb}
\usepackage{graphicx}
\usepackage{geometry}
\usepackage{enumerate}
\newtheorem{thm}{Theorem}[section]
\newtheorem{cor}[thm]{Corollary}
\newtheorem{lem}[thm]{Lemma}

\newtheorem{prop}[thm]{Proposition}
\newtheorem{exa}[thm]{Example}

\theoremstyle{definition}

\numberwithin{equation}{section}

\begin{document}
\vbox{\vskip 1cm}
\title{When an $\mathscr{S}$-closed submodule is a direct summand } \maketitle

\begin{center}
Yongduo Wang\ \ \ \ \ Dejun Wu\\
\vskip 4mm Department of Applied Mathematics, Lanzhou University of Technology\\
Lanzhou 730050, P. R. China\\
E-mail: ydwang@lut.cn\\
 E-mail: wudj2007@gmail.com\\

\end{center}

\bigskip

\centerline{\bf Abstract} It is well known that a direct sum of
CLS-modules is not, in general, a CLS-module. It is proved that if
$M=M_1\oplus M_2$, where $M_1$ and $M_2$ are CLS-modules such that
$M_1$ and $M_2$ are relatively ojective (or $M_1$ is
$M_2$-ejective), then $M$ is a CLS-module and some known results are
generalized. Tercan [8] proved that if a module $M=M_{1}\oplus
M_{2}$ where $M_{1}$ and $M_{2}$ are CS-modules such that $M_{1}$ is
$M_{2}$-injective, then $M$ is a CS-module if and only if $Z_{2}(M)$
is a CS-module. Here we will show that Tercan's claim is not true.

\bigskip
\noindent

\bigskip

\baselineskip=20pt

\section{Introduction}

CS-modules play important roles in rings and categories of modules
and their generalizations have been studied extensively  by many
authors recently. In [3], Goodearl defined an
\emph{$\mathscr{S}$-closed} submodule of a module $M$ is a submodule
$N$ for which $M/N$ is nonsingular. Note that $\mathscr{S}$-closed
submodules are always closed. In general, closed submodules need not
be $\mathscr{S}$-closed. For example, $0$ is a closed submodule of
any module $M$, but $0$ is $\mathscr{S}$-closed in $M$ only if $M$
is nonsingular. As a proper generalization of CS-modules, Tercan
introduced the concept of CLS-modules. Following [8], a module $M$
is called a \emph{CLS-module} if every $\mathscr{S}$-closed
submodule of $M$ is a direct summand of $M$. In this paper, we
continue the study of CLS-modules. Some preliminary results on
CLS-modules are given in Section 1. In Section 2, direct sums of
CLS-modules are studied. It is shown that if $M=M_1\oplus M_2$,
where $M_1$ and $M_2$ are CLS-modules such that $M_1$ and $M_2$ are
relatively ojective, then $M$ is a CLS-module and some known results
are generalized. Tercan [8] proved that if a module $M=M_{1}\oplus
M_{2}$ where $M_{1}$ and $M_{2}$ are CS-modules such that $M_{1}$ is
$M_{2}$-injective, then $M$ is a CS-module if and only if $Z_{2}(M)$
is a CS-module. It is shown that Tercan's claim is not true in this
section.

  Throughout this paper, $R$ is an associative ring with identity and
all modules are unital right $R$-modules. We use $N\leq M$ to
indicate that $N$ is a submodule of $M$. Let $M$ be a module and
$S\leq M$. $S$ is \emph{essential} in $M$ (denoted by $S\leq_e M$)
if for any $T\leq M, S\cap T=0$ implies $T=0$.  A module $M$ is
 \emph{CS} if for any  submodule $N$ of
$M$, there exists a direct summand $K$ of $M$ such that $N\leq_e K$.
 A submodule $K$ of $M$ closed in $M$ if $K$ has no proper essential
 extension in $M$, i.e., whenever $L$ is a submodule of $M$ such that
  $K$ is essential in $L$ then $K=L$. It is well known that $M$ is
CS if and only every closed submodule is a direct summand of $M$.
$Z(M)$($Z_2(M)$) denotes the (second) singular submodule of $M$. For
standard definitions we refer to [3].

\section{Preliminary results}

\begin{lem}([8, Lemma 7])
Any direct summand of a CLS-module is a CLS-module.
\end{lem}

\begin{prop}
A module $M$ is a CLS-module if and only if for each
$\mathscr{S}$-closed submodule $K$ of $M$, there exists a complement
$L$ of $K$ in $M$ such that every homomorphism $f: K\oplus
L\rightarrow M$ can be extended to a homomorphism $g: M\rightarrow
M$.
\end{prop}

\begin{proof}
This is a direct consequence of [7, Lemma 2].
\end{proof}

Following [1], a module $M$ is \emph{$\mathscr{G}$-extending} if for
each submodule $X$ of $M$ there exists a direct summand $D$ of $M$
such that $X\cap D\leq_eX$ and $X\cap D\leq_eD$.

\begin{prop}
Let $M$ be a $\mathscr{G}$-extending module. Then $M$ is a
CLS-module.
\end{prop}

\begin{proof}
Let $N$ be an $\mathscr{S}$-closed submodule of $M$. There exists a
direct summand $D$ of $M$ such that $N\cap D\leq_eN$ and $N\cap
D\leq_eD$. Note that $D/(N\cap D)$ is both singular and nonsingular.
Then $D=N\cap D$ and so $N=D$. Therefore, $M$ is a CLS-module.
\end{proof}

In general, a CLS-module need not be a $\mathscr{G}$-extending
module as the following example shows.

\begin{exa} Let $K$ be a field and $V=K\times K$. Consider the ring
$R$ of $2\times 2$ matrix of the form $(a_{ij})$ with $a_{11},
a_{22}\in K, a_{12}\in V, a_{21}=0$ and $a_{11}=a_{22}$. Following
[8, Example 14], $R_{R}$ is a CLS module which is not a module with
$(C_{11})$. Therefore, $R_{R}$ is not a $\mathscr{G}$-extending
module by [1, Proposition 1.6].
\end{exa}

Applying Proposition 2.3, we will give some examples which are CLS
modules, but not CS-modules as follows.

\begin{exa} Let $M_{1}$ and $M_{2}$ be Abelian groups (i.e., $\mathbb{Z}$-modules) with $M_{1}$ divisible
and $M_{2}=\mathbb{Z}_{p^{^{n}}}$, where $p$ is a prime and $n$ is a
positive integer.  Since $M=M_{1}\oplus M_{2}$ is
$\mathscr{G}$-extending by [1, Example 3.4], it is a CLS module by
Proposition 2.3. But $M$ is not CS, when $M_{1}$ is torsion-free. In
particular, $\mathbb{Q}\oplus \mathbb{Z}_{p^{^{n}}}(n\geq 2,
p=prime)$ is a CLS module, but not CS.
\end{exa}

\begin{exa}
Let $M_{1}$ be a $\mathscr{G}$-extending module with a finite
composition series, $0=X_{0}\leq X_{1}\leq \cdot\cdot\cdot\leq
X_{m}=M_{1}$. Let $M_{2}=X_{m}/X_{m-1}\oplus \cdot\cdot\cdot\oplus
X_{1}/X_{0}$. Since $M=M_{1}\oplus M_{2}$ is $\mathscr{G}$-extending
by [1, Example 3.4], it is a CLS module by Proposition 2.3. But $M$
is not CS in general. In particular, $M\oplus (U/V)$ is a CLS
module, but not CS, where $M$ is a uniserial module with unique
composition series $0\neq V\subset U\subset M$.

\end{exa}

\begin{prop}
Let $M$ be a nonsingular module. Then the following conditions are
equivalent.
\begin{enumerate}[(i)]
\item $M$ is a CS-module.
\item $M$ is a $\mathscr{G}$-extending module.
\item $M$ is a CLS-module.
\end{enumerate}
\end{prop}

\begin{proof}
By [1, Proposition 1.8] and [8, Corollary 5].
\end{proof}

\begin{prop}
Let $M$ be a CLS-module and $X$ be a submodule of $M$. If
$Z(M/X)=0$, then $M/X$ is a CS-module.
\end{prop}

\begin{proof}
Since $M$ is a CLS-module, $X$ is a direct summand of $M$. Write
$M=X\oplus X', X'\leq M$. Then $M/X$ is a CS-module by Lemma 2.1 and
Proposition 2.7.
\end{proof}

\begin{cor}(1, Proposition 1.9)
If $M$ is $\mathscr{G}$-extending, $X\unlhd M$, and $Z(M/X)=0$, then
$M/X$ is a CS-module.
\end{cor}

\begin{cor}(1, Corollary 3.11(i))
Let $M$ be a $\mathscr{G}$-extending module. If $D$ is a direct
summand of $M$ such that $Z(D)=0$, then $D$ is a CS-module.
\end{cor}

\begin{prop}
Let $K\leq_eM$ such that $K$ is a CLS-module and for each
$e^2=e\in$End$(K)$ there exists $\bar{e}^2=\bar{e}\in$End$(M)$ such
that $\bar{e}|_K=e$. Then $M$ is a CLS-module.
\end{prop}

\begin{proof}
Assume $K$ is a CLS-module. Let $X$ be an $\mathscr{S}$-closed
submodule of $M$. Then $K=(X\cap K)\oplus K', K'\leq K$. Let $X\cap
K=eK$, where $e^2=e\in$End$(K)$. By hypothesis, there exists
$\bar{e}^2=\bar{e}\in$End$(M)$ such that $\bar{e}|_K=e$. Since
$K\leq_eM$, $\bar{e}K\leq_e\bar{e}M$. Observe that $\bar{e}M\cap
X\leq_e\bar{e}M$. But $\bar{e}M/(\bar{e}M\cap X)$ is nonsingular.
Hence $\bar{e}M\leq X$. Thus $X=\bar{e}M$ as $\bar{e}K\leq_eX$.
Therefore, $M$ is a CLS-module.
\end{proof}

By analogy with the proof of [2, Corollary 3.14], we can obtain

\begin{cor}
Let $M$ be a module. If $M$ is CLS, then so is the rational hull of
$M$.
\end{cor}

\section{Direct sums of CLS modules}

It is well known that a direct sum of CLS- modules is not, in
general, a CLS-module (see [8]). In this section, direct sums of
CLS-modules are studied. It is shown that if $M=M_1\oplus M_2$,
where $M_1$ and $M_2$ are CLS-modules and $M_1$ and $M_2$ are
relatively ojective, then $M$ is a CLS-module and some known results
are generalized. Tercan [8] proved that if a module $M=M_{1}\oplus
M_{2}$ where $M_{1}$ and $M_{2}$ are CS-modules such that $M_{1}$ is
$M_{2}$-injective, then $M$ is a CS-module if and only if $Z_{2}(M)$
is a CS-module. It is shown that Tercan's claim is not true in this
section.

Let $A,~B$ be right $R$-modules. Recall that $B$ is $A$-ojective [6]
if and only if for any complement $C$ of $B$ in $A\oplus B, A\oplus
B$ decomposes as $A\oplus B=C\oplus A' \oplus B'$ with $A'\leq A$
and $B'\leq B$. $A$ and $B$ are relatively ojective if $A$ is
$B$-ojective and $B$ is $A$-ojective.

\begin{lem}
Let $M=A\oplus B$, where $B$ is $A$-ojective and $A$ is a
CLS-module. If $X$ is an $\mathscr{S}$-closed submodule of $M$ such
that $X\cap B=0$, then $M$ decomposes as $M=D\oplus A'\oplus B'$,
where $A'\leq A, B'\leq B$.
\end{lem}

\begin{proof}
Let $X$ be an $\mathscr{S}$-closed submodule of $M$ with $X\cap
B=0$. Then $M/X$ is nonsingular. Note that $X\cap A$ is an
$\mathscr{S}$-closed submodule of $A$. Hence $X\cap A$ is a direct
summand of $A$. Write $A=(X\cap A)\oplus A_1, A_1\leq A$. By Lemma
2.1 and Proposition 2.7, $A_1$ is a CS-module. Let $K=(X\oplus
B)\cap A$. Then $X\oplus B=K\oplus B$ and $K=(X\cap A)\oplus(K\cap
A_1)$. There exists a closed submodule $A_1'$ of $A_1$ such that
$K\cap A_1\leq_eA_1'$. Then $A_1'$ is a direct summand of $A_1$.
Write $A_1=A_1'\oplus A_1{''}, A_1{''}\leq A_1$. Now $X\oplus
B=K\oplus B=(X\cap A)\oplus(K\cap A_1)\oplus B\leq_e(X\cap A)\oplus
A_1'\oplus B$. Let $N=(X\cap A)\oplus A_1'\oplus B$. Then $X$ is a
complement of $B$ in $N$. Now $B$ is $(X\cap A)\oplus A_1'$-ojective
by [6, Proposition 8].  By [6, Theorem 7], $N=X\oplus A'\oplus B'$,
where $A'\leq(X\cap A)\oplus A_1'$ and $B'\leq B$. Therefore,
$M=X\oplus A'\oplus A_1{''}\oplus B'$, as required.
\end{proof}

\begin{thm}
Let $M=M_1\oplus M_2$, where $M_1$ and $M_2$ are CLS-modules. If
$M_1$ and $M_2$ are relatively ojective, then $M$ is a CLS-module.
\end{thm}

\begin{proof}
Let $X$ be an $\mathscr{S}$-closed submodule of $M$. If $X\cap
M_1=0$, then $X$ is a direct summand of $M$ by Lemma 3.1. Let $X\cap
M_1\neq 0$. Then $X\cap M_1$ is a direct summand of $M_1$. Write
$M_1=(X\cap M_1)\oplus M_1', M_1'\leq M_1$. If $X\cap M_2=0$, then
the result follows by Lemma 3.1. Let $X\cap M_2\neq 0$. Then $X\cap
M_2$ is a direct summand of $M_2$. Write $M_2=(X\cap M_2)\oplus
M_2', M_2'\leq M_2$. Then $X=(X\cap M_1)\oplus(X\cap
M_2)\oplus(X\cap(M_1'\oplus M_2'))$. Note that $M_1'$ and $M_2'$ are
CS-module and $M_1'$ and $M_2'$ are relatively ojective, so
$M_1'\oplus M_2'$ are a CS-module by [6, Theorem 7]. Hence
$X\cap(M_1'\oplus M_2')$ is a direct summand of $M_1'\oplus M_2'$.
Therefore, $M$ is a CLS-module, as desired.
\end{proof}

\begin{cor}([8, Theorem 10]) Let $R$ be a ring and $M$ a right
$R$-module such that $M=M_1\oplus M_2\oplus\cdots\oplus M_n$ is a
finite direct sum of relatively injective modules $M_i, 1\leq i\leq
n$. Then $M$ is a CLS-module if and only if $M_i$ is a CLS-module
for each $1\leq i\leq n$.
\end{cor}

Let $M_1$ and $M_2$ be modules such that $M=M_1\oplus M_2$. Recall
that $M_1$ is $M_2$-ejective [1] if and only if for every submodule
$K$ of $M$ with $K\cap M_1=0$ there exists a submodule $M_3$ of $M$
such that $M=M_1\oplus M_3$ and $K\cap M_3\leq_eK$.

\begin{lem}
Let $A_1$ be a direct summand of $A$ and $B_1$ a direct summand of
$B$. If $A$ is $B$-ejective, then $A_1$ is $B_1$-ejective.
\end{lem}

\begin{proof}
Write $M=A\oplus B,\  A=A_1\oplus A_2$ and $B=B_1\oplus B_2$. First
we prove that $A_1$ is $B$-ejective. Write $N=A_1\oplus B$. Let $X$
be a submodule of $N$ with $X\cap A_1=0$. Then $X\cap A=0$. Since
$A$ is $B$-ejective, there is a submodule $C$ of $M$ such that
$M=A\oplus C$ and $X\cap C\leq_eX$. Hence $N=A_1\oplus
(N\cap(A_2\oplus C))$. Clearly, $X\cap(N\cap(A_2\oplus C))\leq_eX$.
Therefore, $A_1$ is $B$-ejective. Next we prove that $A$ is
$B_1$-ejective. Write $L=A\oplus B_1$. Let $Y$ be a submodule of $L$
with $Y\cap A=0$. Since $A$ is $B$-ejective, there exists a
submodule $D$ of $M$ such that $M=A\oplus D$ and $D\cap Y\leq_eY$.
Then $L=A\oplus(L\cap D)$. Clearly, $Y\cap(L\cap D)\leq_eY$.
Therefore, $A$ is $B_1$-ejective. Thus $A_1$ is $B_1$-ejective.
\end{proof}

\begin{thm}
Let $M=M_1\oplus M_2$, where $M_1$ and $M_2$ are CLS-modules. If
$M_1$ is $M_2$-ejective, then $M$ is a CLS-module.
\end{thm}

\begin{proof}
Let $N$ be an $\mathscr{S}$-closed submodule of $M$. If $N\cap
M_1=0$, then $M_1$ is nonsingular. Since $M_1$ is $M_2$-ejective,
there is a submodule $M_{3}$ of $M$ such that $M=M_1\oplus M_3$ and
$N\cap M_3\leq_eN$. Note that $N/(N\cap M_3)$ is both singular and
nonsingular. Hence $N=N\cap M_3$. Since $M_3\cong M_2$, $M_3$ is a
CLS-module. Clearly, $M_3/N$ is nonsingular. Then $N$ is a direct
summand of $M$. Let $N\cap M_1\neq 0$. Then $N\cap M_1$ is a direct
summand of $M_1$. Write $M_1=(N\cap M_1)\oplus M_1', M_1'\leq M_1$.
Similarly, $M_2=(N\cap M_2)\oplus M_2', M_2'\leq M_2$. Then
$N=(N\cap M_1)\oplus(N\cap M_2)\oplus(N\cap(M_1'\oplus M_2'))$.
Since $M_1$ is $M_2$-ejective, $M_1'$ is $M_2'$-ejective by Lemma
3.4. Note that $M_1'$ and $M_2'$ are $\mathscr{G}$-extending module.
By [1, Theorem 3.1], $M_1'\oplus M_2'$ is $\mathscr{G}$-extending.
Hence $N\cap(M_1'\oplus M_2')$ is a direct summand of $M_1'\oplus
M_2'$. Therefore, $M$ is a CLS-module, as desired.
\end{proof}

\begin{cor}([8, Theorem 9])
Let $M=M_1\oplus M_2$, where $M_1$ and $M_2$ are CLS-modules. If
$M_1$ is $M_2$-injective, then $M$ is a CLS-module.
\end{cor}

\begin{cor}
Let $M=M_1\oplus M_2\oplus\cdots\oplus M_n$ be a finite direct sum.
If $M_i$ is $M_j$-ejective for all $j>i$ and each $M_i$ is a
CLS-module, then $M$ is a CLS-module.
\end{cor}

\begin{proof}
By analogy with the proof of [1, Corollary 3.2].
\end{proof}

\begin{cor}
Let $M=M_1\oplus M_2$. Then
\begin{enumerate}[(i)]
\item If $M_1$ is injective, then $M$ is a CLS-module if and only if
$M_2$ is a CLS-module.
\item If $M_1$ is a CLS-module and $M_2$ is semisimple, then $M$ is a
CLS-module.
\end{enumerate}
\end{cor}

\begin{cor}
A module $M$ is a CLS-module if and only if $M=Z_2(M)\oplus M',
M'\leq M$, where $Z_2(M)$ and $M'$ are CLS-modules.
\end{cor}

\begin{proof}
Let $M$ be a CLS-module. Then $M=Z_2(M)\oplus M', M'\leq M$. By
Lemma 2.1, $Z_2(M)$ and $M'$ are CLS-modules. Conversely, if
$M=Z_2(M)\oplus M', M'\leq M$, then $M'$ is $Z_2(M)$-injective. Now
the result follows by Theorem 3.5.
\end{proof}

\begin{cor}
Let $M=M_1\oplus M_2$, where $M_1$ and $M_2$ are CS-modules. If $M$
is nonsingular and $M_1$ is $M_2$-ejective, then $M$ is a CS-module.
\end{cor}

\begin{proof}
By Proposition 2.7 and Theorem 3.5.
\end{proof}

\begin{cor}
Let $M=M_1\oplus M_2$ be a direct sum of CS-modules $M_1$ and $M_2$,
where $M_2$ is nonsingular. If $M_1$ is $M_2$-ejective and
$Z_2(M_1)$ is $M_2$-injective, then $M$ is a CS-module.
\end{cor}

\begin{proof}
By analogy with the proof of [8, Corollary 11].
\end{proof}

\begin{cor}([4, Theorem 4])
Let $M=M_1\oplus M_2$ be a direct sum of CS-modules $M_1$ and $M_2$,
where $M_2$ is nonsingular. If $M_1$ is $M_2$-injective, then $M$ is
a CS-module.
\end{cor}

\begin{cor}
Let $M=M_1\oplus M_2$ be a direct sum of CS-modules $M_1$ and $M_2$.
If $M_1$ is $M_2$-ejective, $Z_2(M_1)$ is $M_2$-injective and
$Z_2(M_2)$ is $M_1$-injective, then $M$ is a CS-module if and only
if $Z_2(M)$ is a CS-module.
\end{cor}

\begin{proof}
Let $Z_2(M)$ be a CS-module. Then $M=Z_2(M_1)\oplus Z_2(M_1)\oplus
M_1'\oplus M_2'$, where $M_1'\leq M_1$ and $M_2'\leq M_2$. By [6,
Theorem 1], $Z_2(M_1)$ is $M_1'$-injective and $Z_2(M_2)$ is
$M_2'$-injective. Then $Z_2(M)$ is $M_1'\oplus M_2'$-injective.
Since $M_1$ is $M_2$-ejective, $M_1'\oplus M_2'$ is a CS-module by
Corollary 3.10. Hence $M$ is a CS-module by [6, Theorem 1].
\end{proof}

\begin{cor}
Let $M=M_1\oplus M_2$ be a direct sum of CS-modules $M_1$ and $M_2$
such that $M_1$ is $M_2$-injective and $Z_2(M_2)$ is
$M_1$-injective. Then $M$ is a CS-module if and only if $Z_2(M)$ is
a CS-module.
\end{cor}

We close this paper with the following.

A. Tercan [8, Corollary 13] showed that if a module $M=M_{1}\oplus
M_{2}$ where $M_{1}$ and $M_{2}$ are CS-modules such that $M_{1}$ is
$M_{2}$-injective, then $M$ is a CS-module if and only if $Z_{2}(M)$
is a CS-module. The following example shows that this claim is not
true.

\begin{exa} Let $R=\mathbb{Z}$ and $M_{\mathbb{Z}}=\mathbb{Q}\oplus
\mathbb{Z}_{p^{n}}(n\geq 2, p=prime)$. We know that $\mathbb{Q}$ is
$\mathbb{Z}_{p^{n}}$-injective and $\mathbb{Q}$,
$\mathbb{Z}_{p^{n}}$ are uniform modules. Following by [1, Example
3.4], $M$ is not CS. Next we show that $Z_{2}(M)$ is CS. Since
$\mathbb{Q}_{\mathbb{Z}}$ is nonsingular, it is easy to see that
$Z_{2}(M)=Z_{2}(\mathbb{Z}_{p^{n}})$. Since $\mathbb{Z}_{p^{n}}$ is
CS, $Z_{2}(\mathbb{Z}_{p^{n}})$, as a direct summand of
$\mathbb{Z}_{p^{n}}$, is CS.
\end{exa}

\begin{center}

\textbf{REFERENCES}

\end{center}

\vspace{0.5cm}

\footnotesize{

\textbf{1.} E. Akalan, G.F. Birkenmeier and A. Tercan, \emph{Goldie
extending modules}, Comm. Algebra \textbf{37} (2009), 663-683.

\textbf{2.} G.F. Birkenmeier and A. Tercan, \emph{When some
complement of a submodule is a summand}, Comm. Algebra \textbf{35}
(2007), 597-611.

\textbf{3.} K.R. Goodearl, \emph{Ring theory}, Marcel-Dekker, New
York, 1976.

\textbf{4.} A. Harmanci and P.F. Smith, \emph{Finite direct sums of
CS-modules}, Houston J. Math. \textbf{19} (1993), 523-532.

\textbf{5.}~M.A. Kamal and B.J. M\"uller, \emph{Extending modules
over commutative domains}, Osaka J.Math. \textbf{25} (1988) 531-538.

\textbf{6.} S.H. Mohamed and B.J. M\"uller, \emph{Ojective modules},
Comm. Algebra \textbf{30} (2002) 1817-1827.

\textbf{7.} P.F. Smith and A. Tercan, \emph{Continuous and
quasi-continuous modules}, Houston J. Math. \textbf{18} (1992),
339-348.

\textbf{8.} A. Tercan, \emph{On CLS-modules}, Rocky Mountain J.
Math. \textbf{25} (1995), 1557-1564.

}

\end{document}